\documentclass[12pt,a4paper,leqno,verbatim]{amsart}
\newcounter{minutes}
\divide\time by 60
\newcounter{hours}
\multiply\time by 60 \addtocounter{minutes}{-\time}
\usepackage{amssymb}
\usepackage{hyperref}
\usepackage{graphicx}
\usepackage{color}

\date{}
\newfont{\cyrilic}{wncyr10 scaled 1000}
\title[On the inequalities of S\'andor mean]
{On the inequalities of S\'andor mean}

\author[B.A Bhayo]{Barkat Ali Bhayo}
\address{Koulutuskeskus Salpaus (Salpaus Further Education), 15110 Lahti, Finland\\}
\email{bhayo.barkat@gmail.com}

\newcommand{\comment}[1]{}

\swapnumbers
\theoremstyle{plain}

\newtheorem{theorem}[equation]{Theorem}
\newtheorem{theorem a}[equation]{Theorem A}
\newtheorem{theorem b}[equation]{Theorem B}
\newtheorem{lemma}[equation]{Lemma}

\newtheorem{corollary}[equation]{Corollary}

\newtheorem{remark}[equation]{Remark}

\numberwithin{equation}{section}

\begin{document}

%%%%%%KAUNIS K  \K %%%%%%%%%%%%%%
%\font\fFt=eusm10 %scaled 1200
%\font\fFa=eusm7  %scaled 1200
%\font\fFp=eusm5  %scaled 1200
%\def\K{\mathchoice
 %%%%displaystyle
%{\hbox{\,\fFt K}}
%%%%%textstyle
%{\hbox{\,\fFt K}}
%%%%scriptstyle
%{\hbox{\,\fFa K}}
%%%%%scriptscriptstyle
%{\hbox{\,\fFp K}}}
%%%%%%%%%%%%%
%\def\E{\mathchoice
 %%%%displaystyle
%{\hbox{\,\fFt E}}
%%%%%textstyle
%{\hbox{\,\fFt E}}
%%%%scriptstyle
%{\hbox{\,\fFa E}}
%%%%%scriptscriptstyle
%{\hbox{\,\fFp E}}}
%%%%%%%%%%%%%

%
%
\def\thefootnote{}
\footnotetext{ \texttt{\tiny File:~\jobname .tex,
          printed: \number\year-\number\month-\number\day,
          \thehours.\ifnum\theminutes<10{0}\fi\theminutes}
} \makeatletter\def\thefootnote{\@arabic\c@footnote}\makeatother

%%%%%%%%%%%%%%%%%%%%%%%%%%%%%%%%%%%%%%%%%%%%%%%%%%%%%%%%%%%%%%%%%%%%%%%
%%%%%%%%%%%%%%%%%%%%%%%%%%%%%%%%%%%%%%%%%%%%%%%%%%%%%%%%%%%%%%

\begin{abstract}
In this paper author establishes the two sided inequalities for the 
following S\'andor means
$$X=X(a,b)=Ae^{G/P-1},\quad  Y=Y(a,b)=Ge^{L/A-1},$$
and other related means.
\end{abstract}

\maketitle

%%%%%%%%%%%%%%%%%%%%%%%%%%%%%%%%%%%%%%%%%%%%%%%%%%%%%%%%%%%%%%%%%%%%%%
\bigskip
{\bf 2010 Mathematics Subject Classification}: 26D05, 26D15, 26D99.

{\bf Keywords}: Inequalities, means of two arguments, trigonometric and hyperbolic functions, S\'andor mean.

%%%%%%%%%%%%%%%%%%%%%%%%%%%%%%%%%%%%%%%%%%%%%%%%%%%%%%%%%%%%%%%%%%%%%%%%%%%%%%%%%%%%%%%%%%%%%%%%%
%%%%%%%%%%%%%%%%%%%%%%%%%%%%%%%%%  Introduction   %%%%%%%%%%%%%%%%%%%%%%%%%%%%%%%%%%%%%%%%%%%%%%%
%%%%%%%%%%%%%%%%%%%%%%%%%%%%%%%%%%%%%%%%%%%%%%%%%%%%%%%%%%%%%%%%%%%%%%%%%%%%%%%%%%%%%%%%%%%%%%%%%

\section{\bf Introduction}
The study of the inequalities involving the classical means
such as arithmetic mean $A$, geometric mean $G$, identric mean $I$ and logarithmic mean $L$  has been of the extensive interest for several authors, e.g., see \cite{alzer1, alzer2,carlson,hasto,ns1004, ns1004a,sandorc,sandord,sandore,vvu}.

For two positive real numbers $a$ and $b$, the S\'andor mean 
$X(a,b)$ (see \cite{sandortwosharp})
is defined by
$$X=X(a,b)=Ae^{G/P-1},$$
where
$A=A(a,b)=(a+b)/2, G=G(a,b)=\sqrt{ab},$ and
$$P=P(a,b)=\frac{a-b}{2\displaystyle\arcsin\left(\frac{a-b}{a+b}\right)},\quad a\neq b,$$
are the arithmetic mean, geometric mean, and Seiffert mean 
\cite{seiff}, respectively.

Recently, S\'andor \cite{sandornew} introduced a new mean $Y(a,b)$ 
for two positive real $a$ and $b$, which is defined by
$$Y=Y(a,b)=Ge^{L/A-1},$$
where
$$L=L(a,b)=\frac{a-b}{\log(a)-\log(y)},\quad a\neq b$$
is a logarithmic mean. For two positive real numbers $a$ and $b$, the identric mean and harmonic mean are defined by
$$I=I(a,b)=\frac{1}{e}\left(\frac{a^a}{b^b}\right)^{1/(a-b)},\quad a\neq b,$$
and 
$$H=H(a,b)=2ab/(a+b),$$
respectively.
For the historical background and the generalization of these means we refer the reader to see, e.g, \cite{alzer2,carlson,mit,ns1004,ns1004a,sandor611,sandorc,sandord,sandore,sandor1405,sandorF,vvu}. %Generalizations, or related means are studied in %\cite{newmean,ns1004a,ns1004,ns1605,sandor611,sandornew,sandor1405}.
Connections of these means with the trigonometric or hyperbolic inequalities can be found in \cite{barsan,sandora,sandornew,sandore}.

%$$S=S(a,b)=  (a^a b^b)^{1/(a+b)}.$$
For $p\in \mathbb{R}$ and $a,b>$ with $a\neq b$, the $p$th power mean 
$M_p(a,b)$ and $p$th power-type Heronian mean $H_p$(a,b) are define by 

$$
M_p=M_p(a,b)=
\begin{cases}
\displaystyle\left(\frac{a^p+b^p}{2}\right)^{1/p}, & p\neq 0,\\
\sqrt{ab}, & p=0,
\end{cases}
$$
and 
$$
H_p=H_p(a,b)=
\begin{cases}
\displaystyle\left(\frac{a^p+(ab)^{p/2}+b^p}{3}\right)^{1/p}, & p\neq 0,\\
\sqrt{ab}, & p=0,
\end{cases}
$$
respectively.

In \cite{sandornew}, S\'andor  proved inequalities of $X$ and $Y$ means in terms of other classical means, as well as their relations with each other as follows. 

\begin{theorem}\label{sandortheorem} For $a,b>0$ with $a\neq b$, one has
\begin{enumerate}
\item $G<\frac{AG}{P}<X<\frac{AP}{2P-G}<P$,\\
\item $H<\frac{LG}{A}<Y<\frac{AG}{2A-L}<G$,\\
\item $1<\frac{L^2}{IG}<\frac{L\cdot e^{G/L-1}}{G}<\frac{PX}{AG}$,\\
\item $H<\frac{G^2}{I}<\frac{LG}{A}<\frac{G(A+L)}{3A-L}<Y$.
\end{enumerate}
\end{theorem}

In \cite{barsan}, author and S\'andor gave a series expansion of $X$ and $Y$, and proved the following inequalities.
\begin{theorem}\label{sandorbhayo-theorem} For $a,b>0$ with $a\neq b$, one has
\begin{enumerate}
\item $\frac{1}{e}(G+H)< Y<\frac{1}{2}(G+H),$\\
\item $G^2I<IY<IG<L^2,$\\
\item $\frac{G-Y}{A-L} < \frac{Y+G}{2A}<  \frac{3G+H}{4A}  <1,$\\
\item $L<\frac{2G+A}{3}<X<L(X,A)<P<\frac{2A+G}{3}<I,$\\
\item $ 2\left(1- \frac{A}{P}\right)<\log \left(\frac{X}{A}\right) <
 \left(\frac{P}{A}\right)^2.$
\end{enumerate}
\end{theorem}

In \cite{chuet}, Chu et al. proved that the following double inequality 
\begin{equation}\label{89ineq}
M_p<X<M_q
\end{equation}
holds for all $a,b>0$ with $a\neq b$ if and only if $p\leq 1/3$ and
$q\geq \log (2)/(1+\log(2))\approx 0.4093.$

Recently, Zhou et al. \cite{zhouet} proved that for all $a,b>0$ with $a\neq b$,
the following double inequality
$$H_\alpha < X < H_\beta$$
holds if and only if $\alpha \leq 1/2$ and $\beta\geq \log(3)/(1+\log(2))
\approx 0.6488$.

Making contribution to the topic, in this paper author refines some previous results appeared in \cite{ barsan,sandornew} by giving the following theorems.  

\begin{theorem}\label{thm1} For $a,b>0$, we have
\begin{equation}\label{30ineqa}
\alpha G+(1-\alpha)A<X<\beta G+(1-\beta)A,
\end{equation}
with best possible constants $\alpha=2/3\approx 0.6667$ and $\beta =(e-1)/e\approx 0.6321$,

and

\begin{equation}\label{30ineqb}
A+G-\alpha_1 P<X<A+G-\beta_1 P,
\end{equation}
with best possible constants $\alpha_1=1$ and $\beta_1 =\pi(e-1)/(2e)\approx 0.9929$.
\end{theorem}

\begin{theorem}\label{0208thm} For $a,b>0$, we have
$$a\sqrt{GH}<Y<\sqrt{GH},$$
where $a\approx 0.9756$.
\end{theorem}

\begin{theorem}\label{thm3} For $a,b>0$, we have
\begin{equation}\label{today}
\left(\frac{2+G/A}{2+A/G}\right)^3<\frac{H}{A}<
\left(\frac{2+G/A}{2+A/G}\right)^2,
\end{equation}

\begin{equation}\label{ineq0208c}
\frac{G}{L}<\left(\frac{2}{1+A/G}\right)^{2/3}<
\left(\frac{1+G/A}{2}\right)^{2/3}<\frac{P}{A}.
\end{equation}
\end{theorem}

\begin{theorem}\label{thm30} We have
$$(AX)^{1/\alpha_2}<P<(AX^{\beta_2})^{1/(1+\beta_2)},$$
with best possible constants $\alpha_2=2$ and $\beta_2=\log(\pi/2)/\log(2e/\pi)\approx 0.8234.$
\end{theorem}

The first inequality in \eqref{30ineqb} was proved by S\'andor (see \cite[Theorem 2.10]{sandornew}). The left side of \eqref{30ineqa} is less that the left side of \eqref{30ineqb}, which follows from the inequality 
$$P<\frac{2A+G}{3},$$
(see \cite{sandor1405}). Inequalities in \eqref{ineq0208c} refine the inequalities in \cite[Theorem 2.1]{sandornew}.

This paper is organized as follow: In Section 1, we give the introduction and state the main result. In Section 2, some connections of well-known trigonometric and hyperbolic inequalities with the inequalities of classical means are given. Section 3 deals with the lemmas which will be used in the proof of the theorems. Section 4 consists of the proofs of the theorems.
%%%%%%%%%%%%%%%%%%%%%%%%%%%%%%%%%%%%%%%%%
%%%%%%%%%%%%%%%%%%%%%%%%%%%%%%%%%%%%%%%%%%%%
%%%%%%%%%%%%%%%%%%%%%%%%%%%%%%%%%%%%%%%
\section{Connection with trigonometric functions}
For easy reference we recall the following lemma from \cite{barsan,
barsan2}.
\begin{lemma}\label{lemma1} For $a>b>0$, $x\in(0,\pi/2)$ and $y>0$, one has
\begin{equation}\label{ineq1}
\frac{P}{A}= \frac{\sin (x)}{x},\ \frac{G}{A} = \cos(x),\, \frac{H}{A}= \cos(x)^2,\   
\frac{X}{A}= e^{x {\rm cot}(x)-1},  
\end{equation}

\begin{equation}\label{ineq2} 
\frac{L}{G}= \frac{\sinh (y)}{y},\, \frac{L}{A}= \frac{\tanh (y)}{y},\
 \frac{H}{G}= \frac{1}{\cosh (y)},\,  
\frac{Y}{G}= e^{\tanh (y)/y -1}. 
\end{equation}

\begin{equation}\label{ineq3} 
\log\left(\frac{I}{G}\right)=\frac{A}{L}-1,\quad 
\log\left(\frac{Y}{G}\right)=\frac{L}{A}-1.
\end{equation}
\end{lemma}

\begin{remark}\rm  Recently, the following inequality
\begin{equation}\label{ineq5}
e^{(x/\tanh(x)-1)/2}<\frac{\sinh(x)}{x},\quad x>0,
\end{equation}
appeared in \cite[Theorem 1.6]{barsan3}, which is equivalent to 
$$\frac{\sinh(x)}{x}>e^{x/\tanh(x)-1}\frac{x}{\sinh(x)}.$$ By Lemma \ref{lemma1}, this can be written as
$$\frac{L}{G}>\frac{I}{G}\cdot \frac{G}{L}=\frac{I}{L},$$
or 
\begin{equation}\label{ineq6}
L>\sqrt{IG}.
\end{equation}
The inequality \eqref{ineq6} was proved by Alzer \cite{alzer2}. For the convenience of the reader, we write that inequality \eqref{ineq5}
implies the inequality \eqref{ineq6} as follows:

\begin{equation}
\begin{cases}
\displaystyle e^{(x/\tanh(x)-1)/2}<\frac{\sinh(x)}{x}, & x>0,\\
L>\sqrt{IG}.
\end{cases}
\end{equation}

The  Adamovi\'c-Mitrinovic inequality and Cusa- Huygens
inequality \cite{mit} imply the double double inequality for Seiffert mean 
$P$ as follows:
\begin{equation}
\begin{cases}
\displaystyle \cos(x)^{1/3}<\frac{\sin(x)}{x}<\frac{2+\cos(x)}{3}, & x\in(0,\pi/2),\\
\sqrt[3]{A^2G}<P<\displaystyle\frac{2A+G}{3}.
\end{cases}
\end{equation}

The following trigonometric inequalities 
(see \cite[Theorem 1.5]{barsan3}) imply an other double inequality for Seiffert mean $P$,

\begin{equation}\label{0209f}
\begin{cases}
\displaystyle\exp\left(\frac{1}{2}\left(\frac{x}{\tan x}-1\right)\right)
< \displaystyle\frac{\sin x}{x}< \displaystyle
\exp\left(\left(\log\frac{\pi}{2}\right)\left(\frac{x}{\tan x}-1\right)\right)& x\in(0,\pi/2),\\
\sqrt{AX}<P<A\left(\frac{X}{A}\right)^{\log(\pi/2)}.
\end{cases}
\end{equation}
The second mean inequality in \eqref{0209f} was also pointed out by S\'andor 
(see \cite[Theorem 2.12]{sandornew}).

By observing that $A=G^2/H$, we conclude that the hyperbolic version of Adamovi\'c-Mitrinovic and Cusa-Huygens inequalities (see \cite{neusan0509}) imply the inequalities of Leach and Sholander (see \cite{sandore,sandor1405}),
\begin{equation}\label{028a}
\begin{cases}
\displaystyle \cosh(x)^{1/3}<\frac{\sinh(x)}{x}<\frac{2+\cosh(x)}{3}, & x>0,\\
\sqrt[3]{AG^2}<L<\displaystyle\frac{2A+G}{3}.
\end{cases}
\end{equation}

%The mean inequality in \eqref{028a} can also be written as
%$$\left(\frac{1+G/H}{2}\right)^{2/3}<\frac{L}{G}<
%\left(\frac{1+G/H}{2}\right),$$
%replacing $G/A$ by $H/G$.
%
%
%\begin{equation}\label{028a}
%\begin{cases}
%\left(\frac{1+\cos(x)}{2}\right)^{2/3}<\frac{\sin(x)}{x}<
%\left(\frac{1+\cos(x)}{2}\right)^{1/c_1}, \& c_1=\log2/(\log(\pi/2)),\quad x\in(0,\frac{\pi}{2}),\\
%\left(\frac{1+H/G}{2}\right)^{2/3}<\frac{P}{A}<
%\left(\frac{1+H/G}{2}\right)^{c_1}.$$
%\end{cases}
%\end{equation}

\end{remark}

%%%%%%%%%%%%%%%%%%%%%%%%%%%%%%%%%%%
%%%%%%%%% Lemmas %%%%%%%%%
%%%%%%%%%%%%%%%%%%%%%%%%%%%%%%%%
\section{Preliminaries and lemmas}
The following result by Biernacki and Krzy\.z \cite{bier} will
be used in studying the monotonicity of certain power series.

\begin{lemma}\label{lembk}
For $0<R\leq \infty$. Let $A(x)=\sum_{n=0}^\infty a_nx^n$ and 
$C(x)=\sum_{n=0}^\infty c_nx^n$ be two real power series converging on the interval $(-R,R)$. If the sequence
$\{a_n/c_n\}$ is increasing (decreasing) and $c_n>0$ for all $n$, then the function $A(x)/C(x)$ is also
increasing (decreasing) on $(0,R)$.
\end{lemma}

For $|x|<\pi$, the following power series expansions can be found in \cite[1.3.1.4 (2)--(3)]{jef},
\begin{equation}\label{xcot}
x \cot x=1-\sum_{n=1}^\infty\frac{2^{2n}}{(2n)!}|B_{2n}|x^{2n},
\end{equation}

\begin{equation}\label{cot}
\cot x=\frac{1}{x}-\sum_{n=1}^\infty\frac{2^{2n}}{(2n)!}|B_{2n}|x^{2n-1},
\end{equation}
and 
\begin{equation}\label{coth}
{\rm \coth x}=\frac{1}{x}+\sum_{n=1}^\infty\frac{2^{2n}}{(2n)!}|B_{2n}|x^{2n-1},
\end{equation}
where $B_{2n}$ are the even-indexed Bernoulli numbers 
(see \cite[p. 231]{IR}). 
We can get the following expansions directly from (\ref{cot}) and (\ref{coth}),

\begin{equation}\label{cosec}
\frac{1}{(\sin x)^2}=-(\cot x)'=\frac{1}{x^2}+\sum_{n=1}^\infty\frac{2^{2n}}{(2n)!}
|B_{2n}|(2n-1)x^{2n-2},
\end{equation}

\begin{equation}\label{cosech}
\frac{1}{(\sinh x)^2}=-({\rm coth} x)'=\frac{1}{x^2}-\sum_{n=1}^\infty\frac{2^{2n}}{(2n)!}(2n-1)|B_{2n}|x^{2n-2}.
\end{equation}
For the following expansion formula 
\begin{equation}\label{xsin}
\frac{x}{\sin x}=1+\sum_{n=1}^\infty\frac{2^{2n}-2}{(2n)!}|B_{2n}|x^{2n}
\end{equation}
see \cite{li}.

\begin{lemma}\cite[Theorem 2]{avv1}\label{lem0}
For $-\infty<a<b<\infty$,
let $f,g:[a,b]\to \mathbb{R}$
be continuous on $[a,b]$, and differentiable on
$(a,b)$. Let $g^{'}(x)\neq 0$
on $(a,b)$. If $f^{'}(x)/g^{'}(x)$ is increasing
(decreasing) on $(a,b)$, then so are
$$\frac{f(x)-f(a)}{g(x)-g(a)}\quad and \quad \frac{f(x)-f(b)}{g(x)-g(b)}.$$
If $f^{'}(x)/g^{'}(x)$ is strictly monotone,
then the monotonicity in the conclusion
is also strict.
\end{lemma}

\begin{lemma}\label{lemma5} The following function
$$h(x)=\frac{\log(x/\sin(x))}{\log(e^{1-x/\tan(x)}\sin(x)/x)}$$
is strictly decreasing from $(0,\pi/2)$ onto $(\beta_2,1)$, where
$\beta_2=\log(\pi/2)/\log(2e/\pi)\approx 0.8234.$ In particular, 
for $x\in(0,\pi/2)$ we have
$$\left(\frac{e^{x/\tan(x)-1}\sin(x)}{x}\right)^{\beta_2}<\frac{x}{\sin(x)}<\left(\frac{e^{x/\tan(x)-1}\sin(x)}{x}\right).$$
\end{lemma}

\begin{proof} Let 
$$h(x)=\frac{h_1(x)}{h_2(x)}=\frac{\log(x/\sin(x))}{\log(e^{1-x/\tan(x)\sin(x)/x})},$$
for $x\in(0,\pi/2)$. Differentiating with respect to $x$, we get
$$\frac{h'_1(x)}{h'_2(x)}=\frac{1-x/\tan(x)}{(x/\sin(x))^2-1}=
\frac{A_1(x)}{B_1(x)}.$$
Using the expansion formula 
we have
$$A_1(x)=\sum_{n=1}^\infty\frac{2^{2n}2n}{(2n)!}|B_{2n}|x^{2n}=\sum_{n=1}^\infty a_nx^{2n}$$
and 
$$B_1(x)=\sum_{n=1}^\infty\frac{2^{2n}2n}{(2n)!}|B_{2n}|(2n-1)x^{2n}=\sum_{n=1}^\infty b_nx^{2n}.$$
Let $c_n=a_n/b_n=1/(2n-1)$, which is the decreasing in $n\in\mathbb{N}$. Thus, by Lemma \ref{lembk} $h'_1(x)/h'_2(x)$
is strictly decreasing in $x\in(0,\pi/2)$. In turn, this implies by Lemma 
\ref{lem0} that $h(x)$ is strictly decreasing in $x\in(0,\pi/2)$. Applying l'H\^opital rule, we get 
$\lim_{x\to 0}h(x)=1$ 
and $\lim_{x\to \pi/2}h(x)=\beta_2$. This completes the proof.
\end{proof}

%\begin{lemma}\label{lemma1}\cite[Corollary]{barsan} For $a>b>0$, 
%\begin{equation}\label{ineq1}
%\frac{G}{A} = \cos(x),\, \frac{H}{A}= \cos(x)^2,\   \frac{P}{A}= \frac{\sin (x)}{x},\
%\frac{X}{A}= e^{x {\rm cot}(x)-1},  
%\end{equation}
%\begin{equation}\label{ineq2} 
%\frac{L}{G}= \frac{\sinh (y)}{y},\, \frac{L}{A}= \frac{\tanh (y)}{y},\
 %\frac{H}{G}= \frac{1}{\cosh (y)},\,  
%\frac{Y}{G}= e^{\tanh (y)/y -1}. 
%\end{equation}
%
%\begin{equation}\label{ineq3} 
%\log\left(\frac{I}{G}\right)=\frac{A}{L}-1,\quad 
%\log\left(\frac{Y}{G}\right)=\frac{L}{A}-1
%\end{equation}
%where $G=G(a,b)$, $L=L(a,b)$ and $P=P(a,b)$.
%\end{lemma}

\begin{lemma}\label{lemma2} The following function
$$f(x)=\frac{1-e^{x/\tan(x)-1}}{1-\cos(x)}$$
is strictly decreasing from $(0,\pi/2)$ onto 
$((e-1)/e,2/3)$where $(e-1)/e\approx 0.6321$.
In particular, for $x\in(0,\pi/2)$, we have
$$\frac{1}{\log(1+(e-1)\cos(x))}<\frac{\tan(x)}{x}<
\frac{1}{1+\log((1+2\cos(x))/3)}.$$
\end{lemma}

\begin{proof} Write $f(x)=f_1(x)/f_2(x)$, where 
$f_1(x)=1-e^{x/\tan(x)-1}$ and $f_2(x)=1-\cos(x)$ for all $x\in(0\pi/2)$. Clearly,
$f_1(x)=0=f_2(x)$. Differentiating with respect to $x$, we get
$$\frac{f'_1(x)}{f'_2(x)}=\frac{e^{x/\tan(x)-1}}{\sin(x)^3}\left(\frac{x}{\sin(x)^2}-\frac{\cos(x)}{\sin(x)}\right)
=f_3(x).$$ Again
$$f'_3(x)=-\frac{e^{x/\tan(x)-1}}{\sin(x)^3}\left(c(x)-2\right),$$ 
where
$$c(x)=x
\left(\frac{\cos(x)}{\sin(x)}+\frac{x}{\sin(x)^2}\right).$$
In order to show that $f'_3<0$, it is enough to prove that 
$$c(x)>2,$$
which is equivalent to 
$$\frac{\sin(x)}{x}<\frac{x+\sin(x)\cos(x)}{2\sin(x)}.$$ Applying the Cusa-Huygens inequality
$$\frac{\sin(x)}{x}<\frac{\cos(x)+2}{3},$$ we get
$$\frac{\cos(x)+2}{3}<\frac{x+\sin(x)\cos(x)}{2\sin(x)},$$
which is equivalent to $(\cos(x)-1)^2>0$. Thus $f'_3 >0$, clearly $f'_1/f'_2$ is strictly decreasing in $x\in(0,\pi/2)$. By Lemma \ref{lem0}, we conclude that the function $f(x)$ is strictly decreasing in 
$x\in(0,\pi/2)$. The limiting values follows easily. This completes the proof of the lemma.  
\end{proof}

\begin{lemma}\label{lemma3} The following function
$$f_4(x)=
\frac{\sin (x)}{x \left(\cos (x)-e^{x \cot (x)-1}+1\right)}$$
is strictly increasing from $(0,\pi/2)$ onto $(1,c)$, where
$c=2e/(\pi(e-1))\approx 1.0071$. In particular, for $x\in(0,\pi/2)$ we have
$$1+\cos(x)-e^{x/\tan(x)-1}<\frac{\sin(x)}{x}<c(1+\cos(x)-e^{x/\tan(x)-1}).$$
\end{lemma}

\begin{proof} Differentiating with respect to $x$ we get
$$f'_4(x)=\frac{e (x-\sin (x)) \left(e \cos (x)-(x+\sin (x)) e^{x \cot (x)} \csc
   (x)+e\right)}{x^2 \left(e \cos (x)-e^{x \cot (x)}+e\right)^2}.$$
	Let
	
	$$f_5(x)=\log \left((x+\sin (x)) e^{x \cot (x)} \csc (x)\right)-\log (e \cos (x)+e),$$
	We get
	$$f'_4(x)=\frac{2-x \left(\cot (x)+x \csc ^2(x)\right)}{x+\sin (x)},$$
	which is negative by the proof of Lemma \ref{lemma2}, and $\lim_{x\to 0}f_5(x)=0$. This implies that $f'_4(x)>0$, and $f_4(x)$
	is strictly increasing. The limiting values follows easily. This implies the proof.
\end{proof}
%%%%%%%%%%%%%%%%%%%%%%%%%%%%%%%%%%%%%%%%%%%%
%%%%%%%%%%%%%%%%%%%%%%%%%%%%%%%%%%%%%%%%%%%%
\section{Proofs}

\noindent{\bf Proof of Theorem \ref{thm1}.} It follows from Lemma \ref{lemma2} that
$$\frac{e-1}{e}<\frac{1-1/e^{1-x/\tan(x)}}{\cos(x)/e^{1-x/\tan(x)}-1/e^{1-x/\tan(x)}}<\frac{2}{3}.$$
Now we get the proof of \eqref{30ineqa} by utilizing the Lemma \ref{lemma1}.
The proof of \eqref{30ineqb}
follows easily from Lemmas \ref{lemma1} and \ref{lemma2}.
$\hfill\square$

%\begin{theorem}\label{thm2} For $a,b>0$, we have
%$$A+G-\alpha_1 P<X<A+G-\beta_1 P,$$
%with best possible constants $\alpha_1=1$ and $\beta_1 =\pi(e-1)/(2e)\approx 0.9929$.
%\end{theorem}
%
%
%\begin{proof} Proof follows easily from Lemmas \ref{lemma1} and \ref{lemma2}.
%\end{proof}

\bigskip

\noindent{\bf Proof of Theorem \ref{0208thm}.} For the proof of the first inequality see \cite[Theorem 7(2)]{barsan}.
For the validity of the following inequality
$$\frac{\sinh(x)-\cosh(x)}{2x\cosh(x)}<\log\left(\frac{1}{\cosh(x)}\right)$$
see \cite{barsan3}, which is equivalent to 
\begin{equation}\label{ineq0208b}
\sqrt{\cosh(x)}\cdot\exp{\tanh(x)/x-1}<1.
\end{equation}
By Lemma \ref{lemma1} the inequality \eqref{ineq0208b}
implies the proof of the second inequality.

$\hfill\square$

\noindent{\bf Proof of Theorem \ref{thm3}.} Let $g(x)=g_1(x)/g_2(x)$, where
$$g_1(x)=\log\left(\frac{2+1/\cos(x)}{2+\cos(x)}\right),\quad
g_2(x)=\log\left(\frac{1}{\cos(x)}\right),$$
for all $x\in(0,\pi/2)$. Differentiating with respect to $x$ we get
$$\frac{g'_1(x)}{g'_2(x)}=1-\frac{1}{5+2\cos(x)+2/\cos(x)}=g_3(x).$$
The function $g_3(x)$ is strictly increasing in $x\in(0,\pi/2)$, because
$$g'_3(x)=\frac{6\sin(x)^3}{(3+5\cos(x)+\cos(x)^2-\sin(x)^2)^2}
>0.$$ Hence $g'_1(x)/g'_2(x)$ is strictly increasing, and clearly 
$g_1(0)=0=g_2(0)$. Since the function $g(x)$ is stricty increasing by Lemma \ref{lem0}, and we get
$$\lim_{x\to 0} g(x)=\frac{2}{3}<g(x)<1=\lim_{x\to \pi/2}  g(x).$$ 
This implies the proof of \eqref{today}.

Next we consider the proof of \ref{ineq0208c}. By Lemma \ref{lemma1} the following inequality
$$\left(\frac{1+\cos(x)}{2}\right)^{2/3}<\frac{\sin(x)}{x}<
\left(\frac{1+\cos(x)}{2}\right)^{1/c_1}, \quad c_1=\log2/(\log(\pi/2)),\quad x\in(0,\frac{\pi}{2}),$$
implies
\begin{equation}\label{ineq0208d}
\left(\frac{1+H/G}{2}\right)^{2/3}<\frac{P}{A}<
\left(\frac{1+H/G}{2}\right)^{c_1}.
\end{equation}
Similarly,
$$\left(\frac{1+\cosh(x)}{2}\right)^{2/3}<\frac{\sinh(x)}{x}<
\left(\frac{1+\cosh(x)}{2}\right), \quad x>0,$$
gives
\begin{equation}\label{ineq0208e}
\left(\frac{1+G/H}{2}\right)^{2/3}<\frac{L}{G}<
\left(\frac{1+G/H}{2}\right).
\end{equation}
Now the first and the third inequality in \eqref{ineq0208c} are obvious from \eqref{ineq0208d} and \eqref{ineq0208d}. For the proof of the second inequality in \eqref{ineq0208c}, it is enough to prove that 
$$\frac{2}{1+x}<\frac{1+1/x}{2},\quad x>1,$$ 
which holds true, because it can be simplified as
$$(1-x)^2>0.$$
This completes the proof of theorem.
$\hfill\square$

%\begin{theorem}\label{0208thma} For $a,b>0$, we have
%\begin{equation}\label{ineq0208c}
%\frac{G}{L}<\left(\frac{2}{1+A/G}\right)^{2/3}<
%\left(\frac{1+G/A}{2}\right)^{2/3}<\frac{P}{A}.
%\end{equation}
%\end{theorem}
%
%\begin{proof}
%By Lemma \ref{lemma1}, the following inequality
%$$\left(\frac{1+\cos(x)}{2}\right)^{2/3}<\frac{\sin(x)}{x}<
%\left(\frac{1+\cos(x)}{2}\right)^{1/c_1}, \quad c_1=\log2/(\log(\pi/2)),\quad x\in(0,\frac{\pi}{2}),$$
%implies
%\begin{equation}\label{ineq0208d}
%\left(\frac{1+H/G}{2}\right)^{2/3}<\frac{P}{A}<
%\left(\frac{1+H/G}{2}\right)^{c_1}.
%\end{equation}
%Similarly,
%$$\left(\frac{1+\cosh(x)}{2}\right)^{2/3}<\frac{\sinh(x)}{x}<
%\left(\frac{1+\cosh(x)}{2}\right), \quad x>0,$$
%gives
%\begin{equation}\label{ineq0208e}
%\left(\frac{1+G/H}{2}\right)^{2/3}<\frac{L}{G}<
%\left(\frac{1+G/H}{2}\right).
%\end{equation}
%Now the first and the third inequality in \ref{ineq0208c} are obvious from \eqref{ineq0208d} and \eqref{ineq0208d}. For the proof of the second inequality, it is enough to prove that 
%$$\frac{2}{1+x}<\frac{1+1/x}{2},\quad x>1.$$ 
%The above inequality holds true, because it can be simplified as
%$$(1-x)^2>0.$$
%This completes the proof of the theorem.
%\end{proof}

\begin{corollary} For $a,b>0$ with $a\neq b$, we have
$$\frac{I}{L}<\frac{L}{G}<1+\frac{G}{H}-\frac{I}{G}.$$
\end{corollary}

\begin{proof} The first inequality is due to Alzer \cite{alzer2}, while the second inequality follows from the fact that the function
$$x\mapsto \frac{1-e^{x/\tanh(x)-1}}{1-\cosh(x)}:(0,\infty)\to (0,1)$$
is strictly decreasing. The proof of the monotonicity of the function is the analogue to the proof of Lemma \ref{lemma2}.
\end{proof}

\noindent{\bf Proof of Theorem \ref{thm30}.} The proof follows easity from Lemma \ref{lemma5}.
$\hfill\square$

\bigskip

In \cite{Seif2}, Seiffert proved that
\begin{equation}\label{seifineq}
\frac{2}{\pi}A<P,
\end{equation}
for all $a,b>0$ with $a\neq 0$.
As a counterpart of the above result we give the following inequalities.

\begin{corollary}\label{coro89} For $a,b>0$ with $a\neq b$, the following inequalities 
$$\frac{1}{e}A<\frac{\pi}{2e} P <X<P$$
holds true.
\end{corollary}

\begin{proof} The first inequality follows from \eqref{seifineq}. For the proof of the second and the third inequality we write by Lemma \ref{lemma1}
$$f'_5(x)=\frac{X}{P}=\frac{xe^{x/\tan(x)-1}}{\sin(x)}=f_5(x)$$
for $x\in(0,\pi/2)$. Differentiation gives
$$\frac{e^{x/tan(x)-1}}{\sin(x)}\left(1-\frac{x^2}{\sin(x)^2}\right)<0.$$ Hence the function $f_5$ is strictly decreasing in $x$, with 
$$\lim_{x\to 0}f_5(x)=1 \quad {\rm and}\quad \lim_{x\to \pi/2}f_5(x)=\pi/(2e)\approx 
0.5779.$$ This implies the proof.
\end{proof}

\begin{theorem} For $a,b>0$ with $a\neq b$, we have
$$L<M_{1/3}<X<P.$$
\end{theorem}

\begin{proof} For the first inequality see \cite{lin}. The second and third inequality follows from \eqref{89ineq} and Corollary \ref{coro89}, respectively.
\end{proof}

%\begin{question} What are the best possible constants $r$ and $s$, such that
%$$\sqrt[r]{GH}<Y<\sqrt[s]{GH}.$$
%\end{question}
%%%%%%%%%%%%%%%%%%%%%%%%%%%%%%%%%%%%%% References %%%%%%%%%%%%%%%%%%%%%%%%%%%%%%%%%%%%%%%%%%%%
%%%%%%%%%%%%%%%%%%%%%%%%%%%%%%%%%%%%%%%%%%%%%%%%%%%%%%%%%%%%%%%%%%%%%%%%%%%%%%%%%%%%%%%%%%%%%%

\end{document}